\documentclass[12pt]{article}
\title{Resolving Grosswald's conjecture on GRH}
\author{Kevin McGown\footnote{Partially supported by a CSU Chico Internal Research Grant.}\\
Department of Mathematics and Statistics\\ California State University, Chico, CA, USA\\ kmcgown@csuchico.edu\\ \\
Enrique Trevi\~{n}o\\ Department of Mathematics and Computer Science\\ Lake Forest College, Lake Forest, USA\\ trevino@lakeforest.edu \\ \\
Tim Trudgian\footnote{Supported by Australian Research Council DECRA Grant DE120100173.}\\
Mathematical Sciences Institute\\ The Australian National University,
Canberra, Australia\\ timothy.trudgian@anu.edu.au
}
\usepackage{url}
\usepackage{amsthm}
\usepackage{amsmath}
\usepackage{algorithmic}
\usepackage{algorithm}
\usepackage{fullpage}
\usepackage{comment}
\usepackage{amssymb}

\usepackage{booktabs}
\newtheorem{thm}{Theorem}
\newtheorem{lem}{Lemma}
\newtheorem{cor}{Corollary}
\newcommand{\ZZ}{\mathbb Z}

\begin{document}

\maketitle
\begin{abstract}
\noindent
 In this paper we examine Grosswald's conjecture on $g(p)$, the least primitive root modulo $p$. Assuming the Generalized Riemann Hypothesis (GRH), and building on previous work by Cohen, Oliveira e Silva and Trudgian, we resolve Grosswald's conjecture by showing that $g(p)< \sqrt{p} - 2$ for all $p>409$. Our method also shows that under GRH we have $\hat{g}(p)< \sqrt{p}-2$ for all $p>2791$, where $\hat{g}(p)$ is the least prime primitive root modulo $p$. \end{abstract}

\section{Introduction}
Let $g(p)$ denote the least primitive root modulo $p$.
Burgess~\cite{Burgess} showed that $g(p)\ll p^{1/4+\epsilon}$ for any~$\epsilon>0$.
Grosswald~\cite{Grosswald81} conjectured that
\begin{equation}\label{gross}
  g(p) < \sqrt{p}-2,
\end{equation}
for all primes~$p>409$. Clearly, Burgess' result implies (\ref{gross}) for all sufficiently large $p$.  In~\cite{COT} it was shown that (\ref{gross}) is true for all $409<p\leq 2.5\cdot 10^{15}$ and for all $p> 3.38\cdot 10^{71}.$ In this paper, contingent on the Generalized Riemann Hypothesis (GRH) we prove (\ref{gross}) for the remaining values of $p$.  In fact, we prove a stronger result.
Let $\hat{g}(p)$ denote the least prime primitive root modulo $p$.
\begin{thm}\label{Main}
Assume GRH. Then $\hat{g}(p)< \sqrt{p}-2$ for all primes $p>2791$ and $g(p) < \sqrt{p}-2$ for all primes $p>409$.
\end{thm}

We use Theorem \ref{Main} to make the following improvement to Theorem G in \cite{Grosswald81}.

\begin{cor}\label{Core}
Assume GRH.  For all primes $p$, the modular group $\Gamma(p)$ can be generated by the matrix
$
\displaystyle
\begin{pmatrix}
1 & p\\
0 & 1\\
\end{pmatrix}
$
and $p(p-1)(p+1)/12$ canonically chosen hyperbolic elements.
\end{cor}
\begin{proof}
Our Theorem \ref{Main} covers the range of $p$ in Theorem 2 in \cite{Grosswald81}, whence the result follows.
\end{proof}
The layout of this paper is as follows. In \S \ref{ankles} we prove an explicit bound on the least prime primitive root. Using this, we are able to prove Theorem~\ref{Main} for all $p> 10^{43}$. We introduce a sieving inequality in \S \ref{sieves}, which allows us to complete the proof of Theorem \ref{Main}.

Throughout the paper we write $\omega(m)$ to denote the number of distinct prime divisors of~$m$.

\section{An explicit bound on the least prime primitive root}\label{ankles}

\begin{thm}\label{explicit.ankeny}
Assume GRH.
When $p\geq 10^9$, the least prime primitive root $\hat{g}(p)$ satisfies
\begin{equation}\label{yacht}
  \hat{g}(p)\leq\left(\frac{8}{5}\,\left(2^{\omega(p-1)}-1\right)\log p\right)^2
  \,.
\end{equation}
\end{thm}
Before we prove Theorem \ref{explicit.ankeny} we remark that the bound in (\ref{yacht}) is not the sharpest known. Shoup \cite{Shoup} has proved that
$$g(p) \ll \omega(p-1)^{4} (\log(\omega(p-1)) +1)^{4} (\log p)^2,$$
and, as remarked by Martin \cite[p.\ 279]{Martin}, this bound also holds for $\hat{g}(p)$. While this supersedes~(\ref{yacht}) for all sufficiently large primes $p$, the utility of Theorem \ref{explicit.ankeny} is in providing a completely explicit estimate.

We require the following result, which is easily established following the proof of Lemma~2.1 in~\cite{Bach-1997}.
\begin{lem}[Bach]\label{L:bach}
Let $\chi(n)$ denote a non-principal Dirichlet character modulo $p$. When $x\geq 1$, we have:
\begin{equation}\label{pen}
  \left|  \sum_{n<x}
  \Lambda(n)\left(1-\frac{n}{x}\right)
 -\frac{x}{2}
 \right|
 \leq
 \frac{\sqrt{x}}{20}+3,
\end{equation}
and
\begin{equation}\label{pencil}
\left|
  \sum_{n< x}
  \Lambda(n)\chi(n)\left(1-\frac{n}{x}\right)
\right|
\leq c(p,x)\sqrt{x}\log p,
\end{equation}
where
\[
c(p,x):=\frac{2}{3}
\left(
1+\frac{2}{\sqrt{x}}+\frac{3}{x^{3/2}}\right)
\left(
1+\frac{5/3}{\log p}
\right)
+
\frac{\log x+2}{\sqrt{x}\log p}.
\]
\end{lem}

\begin{proof}[Proof of Theorem \ref{explicit.ankeny}]
We may assume $\hat{g}(p)>1099$, or else there is nothing to prove;
indeed, one has $((8/5)\log (10^9))^2\approx 1099.4$.
As in Burgess~\cite[\S4]{Burgess}, we rewrite the function
$$
  f(n)
  =
  \begin{cases}
  1 & \text{if $n$ is a primitive root modulo $p$}\\
  0 & \text{otherwise},
  \end{cases}
$$
as
$$
  f(n)=\frac{\phi(p-1)}{p-1}
  \left\{
  1+
  \sum_{\substack{d|p-1\\d>1}}
  \frac{\mu(d)}{\phi(d)}
  \sum_\chi \chi(n)
  \right\},
$$
where $\sum_\chi$ is taken over all Dirichlet characters modulo $p$ of order $d$.
Suppose that $f(n)=0$ for all primes $n$ (and hence all prime powers) with $n< x$.
We may assume $x\geq 1099$.
We have
\begin{equation}\label{johann}
  1+
  \sum_{\substack{d|p-1\\d>1}}
  \frac{\mu(d)}{\phi(d)}
  \sum_\chi \chi(n)
  =0
\end{equation}
for all prime powers $n< x$.
We multiply (\ref{johann}) by
$\Lambda(n)(1-n/x)$ and sum over all $n< x$.
This gives
$$
  \sum_{n<  x}
  \Lambda(n)\left(1-\frac{n}{x}\right)
  +
  \sum_{\substack{d|p-1\\d>1}}
  \frac{\mu(d)}{\phi(d)}
  \sum_\chi
  \sum_{n< x}
  \Lambda(n)\chi(n)\left(1-\frac{n}{x}\right)
  =0.$$
We apply Lemma~\ref{L:bach} and observe that
$c(p,x)\leq 7/9$ provided
$p\geq 10^9$ and $x\geq 1099$.
Therefore
$$
  \frac{x}{2}\leq
  \frac{\sqrt{x}}{20}+3+\sum_{\substack{d|p-1\\d>1\\\mu(d)\neq 0}}\frac{7}{9}\sqrt{x}\log p,
$$
which implies
$$
\sqrt{x}\leq\frac{1}{10}+\frac{6}{\sqrt{x}}+\frac{14}{9} \left(2^{\omega(p-1)}-1\right)\log p
\leq
\frac{8}{5}\left(2^{\omega(p-1)}-1\right)\log p
\,.
$$
The result follows.
\end{proof}

\begin{cor}\label{ankeny.cor}
Theorem~\ref{Main} is true except possibly when
$p\in(2.5\cdot 10^{15}, 10^{43})$ and $\omega(p-1)\in[7,28]$.
\end{cor}

\noindent\textbf{Proof.}
Unconditionally, Robin (see~\cite[Thm 11]{Robin}) proved
$$
  \omega(n)\leq 1.385\frac{\log n}{\log\log n}
  \,, \quad(n\geq 3).
$$
This, when combined with Theorem~\ref{explicit.ankeny},
shows that $\hat{g}(p)<\sqrt{p}-2$ provided $p\geq 10^{49}$. Hence we may assume $p\leq 10^{49}$.
But the assumption $p\leq 10^{49}$ leads to $\omega(p-1)\leq 31$
and now Theorem~\ref{explicit.ankeny} implies the result provided
$p\leq 10^{47}$.   Repeating this process, we find that $\omega(p-1)\leq 28$ and $p\leq 10^{43}$.
On the other hand, in light of the computations carried out in \cite[\S 4]{COT}, we may assume
$p\geq 2.5\cdot 10^{15}$, in which case Theorem~\ref{explicit.ankeny} proves the result
provided $\omega(p-1)\leq 6$.
\qed

\section{Computations using the sieve}\label{sieves}

In this section we follow closely the argument given in \cite[\S 3]{COT}.
Suppose $e$ is an even divisor of $p-1$.  Let $p_1,\dots,p_s$ be the primes dividing $p-1$ that do not divide $e$.
Set $\delta=1-\sum_{i=1}^s p_i^{-1}$, and set $n=\omega(p-1)$.
In applying our method it is essential to choose $\delta>0$.

\begin{thm}\label{ankeny.with.sieve}
Assume GRH.  If $\hat{g}(p)>x$, then we have:
\begin{equation}\label{cantata}
  \hat{g}(p)\leq \left(
  2 c(p,x)
\left(
2+\frac{s-1}{\delta}
\right)
2^{n-s}
\log p
  \right)^2.
\end{equation}
\end{thm}

We postpone the proof of Theorem~\ref{ankeny.with.sieve} until \S\ref{S:proof}.
From Theorem~\ref{ankeny.with.sieve} we immediately obtain the following corollary
which is more readily applied.

\begin{cor}\label{C:ankeny.with.sieve}
Assume GRH.  If $p\geq p_0$, then
\begin{equation}\label{cantata.redux}
  \hat{g}(p)\leq \left(
  C
\left(
2+\frac{s-1}{\delta}
\right)
2^{n-s}
\log p
  \right)^2
  \,,
\end{equation}
where the constant is given in Table~\ref{Table:constants}.
\end{cor}

\begin{table}[t]
\centering
\begin{tabular}{llllll}
\toprule
$p_0$ & $10^2$ & $10^4$ & $10^{6}$ & $10^{8}$ & $10^{10}$ \\ \cmidrule{2-6}
$C$ & $2.1127$ & $1.6821$ & $1.5556$ & $1.496$ & $1.4614$ \\
\midrule
$p_0$ & $10^{12}$ & $10^{14}$ & $10^{16}$ & $10^{18}$ & $10^{20}$ \\ \cmidrule{2-6}
$C$ & $1.4389$ & $1.4231$ & $1.4114$ & $1.4023$ & $1.3952$\\
\bottomrule
\end{tabular}
\caption{Values of $C$ for various choices of $p_0$\label{Table:constants}}
\end{table}

\begin{proof}
When $p\geq p_0$, the right-hand side of~(\ref{cantata.redux}) is bounded below by $x:=((4/3)4\log p_0)^2$.
Hence we may assume $\hat{g}(p)>x$, or else there is nothing to prove.
Now Theorem~\ref{ankeny.with.sieve} establishes the result with $C(p_0):=2 c(p_0,x)$.
\end{proof}

Our proof of Theorem~\ref{Main} will apply Theorem~\ref{ankeny.with.sieve} directly, but
we have included Corollary~\ref{C:ankeny.with.sieve} as it may have application elsewhere.

\begin{cor}\label{C:ankeny.with.sieve2}
Theorem~\ref{Main} is true except possibly when
$\omega(p-1)\in[12,13,14]$.
\end{cor}

\begin{proof}
In light of Corollary~\ref{ankeny.cor}, we may assume $7\leq n\leq 28$ and $2.5\cdot 10^{15}<p<10^{43}$.
We have the obvious lower bound $p-1\geq q_1\dots q_n$, where $q_{i}$ denotes the $i$th prime, and hence we may assume
\[
  p>\max\left\{1+\prod_{i=1}^n q_i
  \,,\;
  2.5\cdot 10^{15}\right\}.
\]
For example, when $n=15$, this leads to $p>6.1\cdot 10^{17}$.

Set $x=\sqrt{p}-2$.  We may assume $\hat{g}(p)>x$ or else there is nothing to prove.  Hence the conclusion of Theorem~\ref{ankeny.with.sieve} holds.
Now each choice of $s$ allows us to show that
\begin{equation}\label{GC}
  \hat{g}(p)<\sqrt{p}-2
\end{equation}
holds when $p$ is larger than an explicitly computable value;
one simply bounds the right-hand side of~(\ref{cantata}) from above, using\footnote{
One helpful fact --- if the right-hand side of~(\ref{cantata}) is less than $\sqrt{p}-2$ for some $p$, then the same is true for all larger $p$.}
$\delta\geq 1-\sum_{i=n-s+1}^n p_i$. Of course, we then choose the value of $s$ that gives the best result.  For example, when $n=15$ we find that $s=12$ leads to
\[
\delta\geq 1-(1/7+1/13+\dots+1/47)>0.3717
\]
and therefore the right-hand side of (\ref{cantata}) is less than $\sqrt{p}-2$ provided $p\geq 3.2\cdot 10^{16}$; hence any exception to $(\ref{GC})$ must satisfy
$p<3.2\cdot 10^{16}$.
Notice that because our lower and upper bounds on any potential exceptions overlap, this proves the result when $n=15$.
The best choice turns out to be $s=n-2$ when $7\leq n\leq 12$ and $s=n-3$ when $13\leq n\leq 28$.
In fact, this is enough to prove
(\ref{GC}) except when $n=12,13,14$.  The lower bound of $1+\prod_{i=1}^n q_i$ does the job when $15\leq n\leq 28$
and the lower bound of $1.6\cdot 10^{15}$ does the job when $7\leq n\leq 11$.

\end{proof}

%

\subsection{An algorithm}

In order to deal with the cases when $n=12,13,14$, we introduce an algorithm.  Before giving the algorithm, we explain the main idea.

Suppose $n=14$.  Using the idea presented in the proof of Corollary~\ref{C:ankeny.with.sieve2}, we find that any exception to Grosswald's conjecture must lie in the interval $(1.30\cdot 10^{16}, 1.71\cdot 10^{16})$.  In principle one could check the conjecture directly for each prime $p$ in this interval, but the size of the interval makes this prohibitive.  There are $2.05\cdot 10^{15}$ odd values of $p$ to consider.  (Of course many of these are not prime, but there are still about $10^{14}$ primes in this interval.)  Instead, we break the problem into cases depending upon which primes divide $p-1$.  We already know that $2$ divides $p-1$, so we start with the prime $3$.  If $3$ divides $p-1$, then we have one third as many values of $p$ to check, roughly $6.83\cdot 10^{14}$ values of $p$.  If $3$ does not divide $p-1$, then this leads to an improved lower bound on $p$, as well as an improved lower bound on $\delta$ and hence an improved upper bound on $p$; in short, the interval under consideration shrinks.  In this particular case, the interval shrinks to $(2.04\cdot 10^{17}, 1.45\cdot 10^{15})$, which is empty, so there is nothing to check.

Returning to the case where $3$ divides $p-1$, the number of exceptions is still quite large.  However, we may consider whether $5$ divides $p-1$.  We continue in this way until the number of possible values of $p$ under consideration is less than $10^6$.  At that point we go through the list and throw out all values of $p$ except those where $p$ is prime,
$\omega(p-1)=14$, and $p-1$ satisfies the given divisibility conditions (depending upon which sub-case we are considering).
We append these exceptional values of $p$ to a list and continue this recursive procedure until we have exhausted all possibilities.
One can easily find the least prime primitive root for the list of exceptions via standard methods and check the conjecture directly.
When $n=14$, this algorithm takes only $7$ seconds (on a 2.7 GHz iMac)
to complete and the list of exceptions is empty, so there is nothing further to check.  The number of exceptions for other values of $n$ is given in Table~\ref{Table:except}.

\begin{table}[ht]
  \caption{Number of exceptions for $n=12,13,14$}
  \centering
  \begin{tabular}{|c||c|c|c|}
    \hline
    $n$ & $12$ & $13$ & $14$\\
    \hline
    \# of exceptions & $61,114$ & $6,916$ & $0$\\
  \hline
  \end{tabular}
  \label{Table:except}
\end{table}


For completeness, we give the pseudocode for our recursive algorithm.  Suppose $X\cup Y$ consists of the first $k$ primes for some $k\in\ZZ_{\geq 0}$.  Algorithm~\ref{A:1} will verify $\hat{g}(p)<\sqrt{p}-2$ when $\omega(p-1)=n$ under the assumption that
$q$ divides $p-1$ for all $q\in X$ and $q$ does not divide $p-1$ for all $q\in Y$.
More precisely, rather than verifying the conjecture for all $p$, the algorithm will generate a manageable list of possible exceptions which can be checked individually, as described above.
The sets $X$ and $Y$ are allowed to be empty,
although in practice we may always assume $2\in X$.
(Running the algorithm with $n=14$, $X=\{2\}$, $Y=\emptyset$ will carry out the computation described above.)

\begin{algorithm}\label{Alg:1}
\caption{Grosswald(n,X,Y)}
\label{A:1}
\begin{algorithmic}[1]
\STATE $L:=$ first $n$ primes not in $Y$
\STATE $lower :=\max\{product(L)+1,\; 2\cdot 10^{15}\}$
\STATE $upper:=0$
\FOR{$s\in\{1,\dots,n-1\}$}
	\STATE $M:=$ largest $s$ primes in $L$
	\STATE $\delta:=1-\sum_{q\in M}\frac{1}{q}$
         \IF{$\delta\leq 0$}
		\STATE \textbf{continue}
	\ENDIF
	\STATE Choose $p$ large enough so that when $x=\sqrt{p}-2$, we have:
	$$\left(2 c(p,x) \left(2+\frac{s-1}{\delta}\right) 2^{n-s}\log(p)\right)^2<x$$
	\IF{$upper=0$ \OR $p<upper$}
	   \STATE $upper:=p$; $\hat{s}:=s$; $\hat{\delta}:=\delta$; $\hat{M}:=M$
	\ENDIF
\ENDFOR
\STATE print($n$, $\hat{s}$, $\hat{\delta}$, $\hat{M}$, $lower$, $upper$)
\STATE prodX := product(X)
\STATE $enum := (upper-lower)/prodX$
\IF{$enum \leq 0$}
\STATE print(``Nothing more to check.'')
\ELSIF{$enum> 10^6$}
\STATE print(``Break into cases.'')
\STATE $q:=$smallest prime not in $X\cup Y$
\STATE $Grosswald(n,X\cup\{q\},Y)$
\STATE $Grosswald(n,X,Y\cup\{q\})$
\ELSE
\STATE print(``Find all exceptions.'')
\STATE $I:=[(lower-1)/prodX,\,(upper-1)/prodX]\cap\ZZ$
\STATE Sieve out elements of $I$ divisible by primes in $Y$
\FOR{$k\in I$}
	\STATE $p:=k*prodX+1$
 		\IF{is\_pseudoprime($p$) and length(prime\_divisors($p-1$))=$n$}
	  \STATE append $p$ to global list of exceptions
	\ENDIF
\ENDFOR
\ENDIF
\end{algorithmic}
\end{algorithm}

\begin{proof}[Proof of Theorem~\ref{Main}]
We have implemented Algorithm~\ref{A:1} in Sage.  Running our code on $n=12,13,14$, including finding the least prime primitive root and checking the conjecture directly for the list of $68,030$ exceptions, takes about $4.5$ minutes.
In light of Corollary~\ref{C:ankeny.with.sieve2}, this proves the theorem.
\end{proof}

\subsection{Proof of Theorem~\ref{ankeny.with.sieve}}\label{S:proof}

Let $p$ be an odd prime.  Let $e$ be an even divisor of $p-1$.  We say that $n$ is $e$-free if
the equation $y^d\equiv n\pmod{p}$ is insoluble for all divisors $d$ of $e$ with $d>1$.  An integer is a primitive root
if and only if it is $(p-1)$-free.  We define the function
$$
  f_e(n)=
  \begin{cases}
  1 &\text{ if $n$ is $e$-free}\\
  0 & \text{ otherwise}.\\
  \end{cases}
$$
Define the multiplicative function $\theta(n)=\phi(n)/n,$ where $\phi(n)$ is Euler's totient function.
We rewrite $f_{e}(n)$ as
$$
  f_e(n)=\theta(e)\left\{
  1+\sum_{\substack{d|e\\d>1}}
  \frac{\mu(d)}{\phi(d)}
  \sum_\chi
  \chi(n)
  \right\}.
$$
We see
$$
  \sum_{i=1}^s
  f_{p_i e}(n)-(s-1)f_e(n)
  \;
  \begin{cases}
  =1 &\text{ if $n$ is $(p_i e)$-free for all $i$}\\
  \leq 0 & \text{ otherwise.}\\
  \end{cases}
$$
Thus
\begin{eqnarray}\label{lid}
f_{p-1}(n)
&\geq&
\sum_{i=1}^s
f_{p_i e}(n)-(s-1)f_e(n) \notag
\\
&=&
\sum_{i=1}^s
\left(f_{p_i e}(n)-\theta(p_i)f_e(n)\right)
+\delta f_e(n)
\,.
\end{eqnarray}
Observe that
\begin{eqnarray}\label{eraser}
f_{p_i e}(n)-\theta(p_i)f_e(n)
&=&
\theta(p_i e)\sum_{\substack{d|p_i e\\
d\nmid e\\ d>1}}
\frac{\mu(d)}{\phi(d)}\sum_{\chi_d} \chi(n) \notag
\\
&=&
\theta(p_i e)
\sum_{d|e}
\frac{\mu(p_i d)}{\phi(p_i d)}\sum_{\chi_{p_i d}} \chi(n).
\end{eqnarray}
Inserting (\ref{eraser}) into (\ref{lid}) leads to
\begin{eqnarray}\label{ruler}
1
+\frac{1}{\delta}\sum_{i=1}^s
\theta(p_i)
\sum_{\substack{d|e}}
\frac{\mu(p_i d)}{\phi(p_i d)}
\sum_\chi
\chi(n)
+
\sum_{\substack{d|e\\d>1}}
\frac{\mu(d)}{\phi(d)}
\sum_\chi
\chi(n)
\leq
\frac{f_{p-1}(n)}{\delta\theta(e)}.
\end{eqnarray}
Suppose $f_{p-1}(n)=0$ for all primes $n$ (and hence all prime powers) with $n<x$.
We multiply (\ref{ruler}) by $\Lambda(n)(1-n/x)$, sum over all $n<x$, which yields\begin{eqnarray*}
&&
\sum_{n<x}
\Lambda(n)\left(1-\frac{n}{x}\right)
\\
&&
\qquad
+
\frac{1}{\delta}
\sum_{i=1}^s
\theta(p_i)
\sum_{\substack{d|e}}
\frac{\mu(p_i d)}{\phi(p_i d)}
\sum_\chi
\sum_{n<x}\chi(n)
\Lambda(n)\left(1-\frac{n}{x}\right)
\\
&&
\qquad
+
\sum_{\substack{d|e\\d>1}}
\frac{\mu(d)}{\phi(d)}
\sum_\chi
\sum_{n<x}\chi(n)
\Lambda(n)\left(1-\frac{n}{x}\right)
\leq 0.
\end{eqnarray*}
We write
$$
  \sum_{n<x}
  \Lambda(n)\left(1-\frac{n}{x}\right)
  =\frac{x}{2}+G(x)
  \,,
$$
and use the estimates
in (\ref{pen}) and  (\ref{pencil})
to obtain
\begin{eqnarray*}
&&
\frac{x}{2}+G(x)
\\
\qquad
&&
\leq
\frac{1}{\delta}
\sum_{i=1}^s
\theta(p_i)
\sum_{\substack{d|e\\\mu(p_i d)\neq 0}}
c(p,x)\sqrt{x}\log p
+
\sum_{\substack{d|e\\d>1\\\mu(d)\neq 0}}
c(p,x)\sqrt{x}\log p
\\
\qquad
&&
\leq
c(p,x)\sqrt{x}\log p
\left[
\left(
1+
\frac{1}{\delta}\sum_{i=1}^s \theta(p_i)
\right)
2^{n-s}
-
1
\right]
\\
\qquad
&&
=
c(p,x)\sqrt{x}\log p
\left[
\left(
2+\frac{s-1}{\delta}
\right)
2^{n-s}
-
1
\right]
\,.
\end{eqnarray*}
This leads to
\begin{eqnarray*}
\sqrt{x}
&\leq&
\frac{2|G(x)|}{\sqrt{x}}
+
2 c(p,x)
\left[
\left(
2+\frac{s-1}{\delta}
\right)
2^{n-s}
-1\right]
\log p
\\
&\leq&
\frac{1}{10}
+
\frac{6}{\sqrt{x}}
+
2 c(p,x)
\left[
\left(
2+\frac{s-1}{\delta}
\right)
2^{n-s}
-1\right]
\log p
\\
&\leq&
2 c(p,x)
\left(
2+\frac{s-1}{\delta}
\right)
2^{n-s}
\log p
-\frac{191}{90}-\frac{4}{3}\log p
\end{eqnarray*}
The result follows.
\qed

\end{document}